\documentclass[a4,12pt,reqno]{amsart}
\usepackage{amsthm,amsmath,amssymb}
\usepackage{latexsym}
\usepackage{enumerate}

 \newtheorem{theorem}{Theorem}[section]
 \newtheorem{lem}[theorem]{Lemma}
 \newtheorem{prop}[theorem]{Proposition}
 
\theoremstyle{remark} 
 \newtheorem{rem}{Remark}
  
\makeatletter
\@addtoreset{equation}{section}

\makeatother

\begin{document}
\title[Inverse scattering for NLS]{
Inverse scattering 
for the nonlinear Schr\"{o}dinger equation 
with the Yukawa potential}
\author[H. Sasaki]{
Hironobu Sasaki$^*$}
\thanks{$^*$Supported by Research Fellowships of 
the Japan Society for the Promotion of Science 
for Young Scientists.} 
\address{
Department of Mathematics, Osaka University, 
563-0043, Japan.}
\email{hisasaki@cr.math.sci.osaka-u.ac.jp}
\maketitle
\begin{abstract}
We study 
the inverse scattering problem 
for the three dimensional 
nonlinear Schr\"{o}dinger equation 
with the Yukawa potential. 
The nonlinearity of the equation is nonlocal.
We reconstruct the potential 
and the nonlinearity 
by the knowledge of the 
scattering states.  
Our result is applicable to 
reconstructing the nonlinearity of
the semi-relativistic Hartree equation.
\end{abstract}
\section{Introduction}\label{sec:intro}
We consider the inverse scattering problem 
for the three dimensional 
nonlinear Schr\"{o}dinger equation 
\begin{equation}\label{NLS}
i\partial_t u +
\Delta u +
Q_0 \frac{\exp(-\mu_0 r)}{r}u - 
\left(
Q_1 \frac{\exp(-\mu_1 r)}{r} \ast |u|^2
\right) u =0 \tag{NLS}
\end{equation}
in $\mathbb{R}\times\mathbb{R}^3$. 
Here, $u$ is a complex-valued unknown function of 
$(t,x)\in \mathbb{R}\times \mathbb{R}^3$, 
$\partial_t =\partial /\partial t$,  
$\Delta$ is the Laplacian in $\mathbb{R}^3$, 
$r=|x|$, 
$Q_0,Q_1 \in \mathbb{R}^3$, 
$\mu_0,\mu_1>0$ and 
$\ast$ is the convolution in the space variables.  
Recall that the functions 
\[
V_j:=-Q_j\dfrac{\exp(-\mu_j r)}{r},
\quad   
j=0,1, 
\]
are said to be the Yukawa potential.
The equation (\ref{NLS}) is approximately 
derived from the generalization of 
the electronic Hamiltonian 
for an $N$-electron atom in a plasma: 
\[
\mathbf{H}=
-\frac{1}{2}\sum_{j=1}^N\Delta_j 
-\sum_{j=1}^N\frac{Z\exp(-\mu_0 |x^j|)}{|x^j|}
+\sum_{j>k}^N\frac{\exp(-\mu_1 |x^j-x^k|)}{|x^j-x^k|},
\]  
where 
$x^j\in \mathbb{R}^3$ 
is the space variables for the $j$-th particle, 
$\Delta_j$ is the Laplacian with respect to $x^j$, 
$Z$ is the nuclear charge 
and  
$\mu_l$, $l=0,1$, are parameters depending on 
the density and the temperature of the plasma 
(see, e.g., Mukherjee--Karwowski--Diercksen \cite{MKD}).

In order to mention 
the inverse scattering problem, 
we introduce the definition of 
the scattering operator 
for the nonlinear evolution equation
\begin{equation}\label{DE}
i\partial_t v(t)+J(v(t))=f(v(t)), 
\quad
t\in\mathbb{R}, 
\end{equation}
where 
$v$ is a complex-valued function 
on the Hilbert space $X$, 
$J$ is a self-adjoint operator on $X$ 
and 
$f$ is a perturbed term.  
Let $B(\delta;X)$ be the set of all 
$\phi\in X$ with $\| \phi\|_X \le \delta$.
\textit{The scattering operator} $S$ 
is defined by the mapping 
\[
S:B(\delta;X)\ni \phi_- \mapsto \phi_+ \in X 
\]
if the following condition holds for some $\delta>0$ and 
some $Z\subset C(\mathbf{R};X)$: \\

\textit{
For any $\phi_-\in B(\delta;X)$, 
there uniquely exists $v\in Z$ 
such that $v$ is a time-global solution to (\ref{DE}) 
and 
satisfies  
\[
\lim_{t\to-\infty}
\| v(t)-e^{itJ}\phi_-\|_X =0.
\]
Furthermore, 
there uniquely exists  $\phi_+\in X$ 
such that 
\[
\lim_{t\to\infty}
\| v(t)-e^{itJ}\phi_+\|_X =0.
\]
}
We remark that $e^{itJ}\phi$ is 
a solution to the Cauchy problem for 
\begin{eqnarray*}
\left\{
  \begin{array}{rl}
   i\partial_t v(t)+J(v(t))=0,&
   \quad t\in\mathbb{R},\\
    v(0)=\phi .&    \\
  \end{array}
\right.
\end{eqnarray*}

The inverse scattering problem for 
the equation (\ref{DE})
is 
to recover the perturbed term $f$  
by applying the knowledge of the 
scattering operator $S$. 
Before we treat (\ref{NLS}), 
we first review 
the inverse scattering problem 
for the Schr\"{o}dinger equation 
with power nonlinearity briefly. 
Strauss \cite{Strauss 1973} 
considered the nonlinear Schr\"{o}dinger equation 
\[
i\partial_t u +\Delta u =V(x)|u|^{p-1}u, 
\quad (t,x)\in \mathbb{R}\times \mathbb{R}^n.
\]
Suppose that 
$p$ is an integer satisfying 
\[
\left\{
  \begin{array}{cl}
    p>4   &\text{if $n=1$},\\
    p>3   &\text{if $n=2$},\\
    p\ge 3&\text{if $n\ge 3$},\\
  \end{array}
\right.
\]
and $V(x)$ is real-valued continuous and bounded, 
whose derivatives up to order $l>3n/4$ are bounded. 
Then the scattering operator $S$ is well-defined. 
It was shown that $V(x)$ is recovered from 
the scattering operator by the following way: 
For $s\in\mathbb{R}^n$, 
let $H^s(\mathbb{R}^n)$ be 
the Sobolev space 
$(1-\Delta)^{-s/2}L^2(\mathbb{R}^n)$. 
For any 
$\phi\in H^1(\mathbb{R}^n)\cap L^{1+1/p}(\mathbb{R}^n)$, 
we have
\begin{equation}\label{MS}
V(x_0)=
\frac{\displaystyle{\lim_{\alpha\to 0}\alpha^{-(n+2)}I[\phi_{\alpha,x_0}]}}
{\displaystyle{
\int_{\mathbb{R}}\int_{\mathbb{R}^n}
| e^{it\Delta}\phi(x)|^{p+1} 
dxdt}} ,
\end{equation}
where 
$\phi_{\alpha,x_0}(x)=\phi(\alpha^{-1}(x-x_0))$, 
$\alpha>0,x,x_0\in \mathbb{R}^n$ 
and  
\[
I[\phi]=\lim_{\varepsilon\to 0}  
\frac{i}{\varepsilon^p}
\big\langle 
(S-id)(\varepsilon\phi),\phi
\big\rangle_{L^2(\mathbb{R}^n)} .
\]  
The above limit is called 
the small amplitude limit.  
Later, Weder 
\cite{Weder97,Weder00-2,Weder01-1,Weder01-2,Weder05}
proved that  
a more general class of nonlinearities 
is uniquely reconstructed, 
and moreover, 
a method is given for the unique reconstruction 
of the potential that acts as a linear operator 
and that this problem was not considered in \cite{Strauss 1973}.

Unfortunately, 
the above methods to obtain 
the reconstruction formulas  
are not applicable to the case (\ref{NLS}) 
even if $Q_0=0$.
The essential point to prove the formula (\ref{MS}) 
is the change of variables in the following integral: 
\[
I[\phi]=
\int_{\mathbb{R}}\int_{\mathbb{R}^n}
V(x)|e^{it\Delta}\phi(x)|^{p+1}
dxdt .
\]
By changing variable $x$ by 
$\alpha^{-1}(x-x_0)$, 
we have 
\[
I[\phi_{\alpha,x_0}]=
\alpha^{n+2}
\int_{\mathbb{R}}\int_{\mathbb{R}^n}
V(x_0+\alpha x)
|e^{it\Delta}\phi(x)|^{p+1}
dxdt .
\]  
Therefore, 
as $\alpha\to 0$, 
we can take the value $V(x_0)$ 
from the inside integral.
Applying the same method to 
(\ref{NLS}) with $Q_0=0$,  
we obtain 
\begin{align*}
I[\phi_{\alpha,x_0}]=&
\int_{\mathbb{R}}\int_{\mathbb{R}^n}
\left(
V\ast |e^{it\Delta}\phi_{\alpha,x_0}|^2
\right) (x)
|e^{it\Delta}\phi_{\alpha,x_0}(x)|^2
dxdt \\
=&
\alpha^{2n+2}
\int_{\mathbb{R}}\int_{\mathbb{R}^n}
\left(
V(\alpha \cdot)\ast |e^{it\Delta}\phi|^2
\right) (x)
|e^{it\Delta}\phi(x)|^2
dxdt ,
\end{align*}
where $V(x)=Q_1\dfrac{\exp(-\mu_1 r)}{r}$.
Since the integral 
\[
\int_{\mathbb{R}}\int_{\mathbb{R}^n}
\left(
V(0)\ast |e^{it\Delta}\phi|^2
\right) (x)
|e^{it\Delta}\phi(x)|^2
dxdt
\]
does not converge, 
we can not make $\alpha$ tend to infinity. 

We next review 
the inverse scattering problem 
for the nonlinear Schr\"{o}dinger equation 
with a cubic convolution 
\begin{equation}\label{Hartree}
i\partial_t u +\Delta u +\widetilde{V}(x)u=F^\sigma(u), 
\quad (t,x)\in \mathbb{R}\times \mathbb{R}^n.
\end{equation}
Here,  
$\widetilde{V}:\mathbb{R}^3\to\mathbb{C}$ 
is measurable and  
satisfies some suitable condition,  
\begin{eqnarray*}
F^\sigma(u)= 
\lambda(x)(|\cdot|^{-\sigma}\ast |u|^2)u 
\end{eqnarray*}
and $\lambda \in C^1(\mathbf{R}^n)\cap W^1_\infty(\mathbf{R}^n)$.
It was proved by 
Watanabe \cite{Watanabe1}
that 
if $\sigma$ is a given number, 
then we can reconstruct 
$V$ and $\lambda$ 
by the knowledge of the scattering operator.
Watanabe \cite{Watanabe3} determined $\sigma$ 
of the term $F^\sigma$ 
if $\widetilde{V}\equiv 0$ and  
$\lambda(x)$ is a non-zero constant function. 
Under the condition $\widetilde{V}\equiv 0$, 
Sasaki \cite{Sasaki 2007} proved that   
$\sigma$ of $F^\sigma$  
can be determined 
even if $\lambda_j$ is not a constant. 
In fact, $\sigma$ is given by 
\begin{align}
\sigma&=
2n+2-
\lim_{\alpha\to 0}\ln
\dfrac{|T[\phi_{e\alpha}]|}{
|T[\phi_\alpha]|+\alpha^{2n+2}},\label{rec:sasaki}\\
T[\phi]&=
\lim_{\varepsilon\to 0}  
\frac{i}{\varepsilon^3}
\big\langle 
(S-id)(\varepsilon\phi),\phi
\big\rangle_{L^2(\mathbb{R}^n)},\nonumber
\end{align}
where 
$e$ is the base of the natural logarithm, 
$\phi\in H^1(\mathbb{R}^n)\setminus \{ 0\}$, 
$\phi_\alpha = \phi_{\alpha,0}$ 
and 
$S$ is the scattering operator. 
For other results of 
the inverse scattering problem for (\ref{Hartree}), 
see Watanabe \cite{Watanabe2,Watanabe4} and 
Sasaki--Watanabe \cite{Sasaki-Watanabe}.

As we mention before, 
we study the inverse scattering problem 
for (\ref{NLS}). 
Remark that we can not directly apply 
the known results 
to recovering the functions $V_j$, $j=0,1$.
Our goal in this paper 
is to give a formula 
for determining the parameter $Q_j$ and $\mu_j$, $j=0,1$,  
by using the knowledge of the scattering operator  
for (\ref{NLS}) given by Theorem \ref{thm:direct} below.

We now define some 
notation  
which will be used later. 
Let $\mathbb{Z}_{\ge 0}$ be 
the set of all non-negative integers. 
For $\Omega\subset \mathbb{R}^n$, 
let $C^\infty_c(\Omega)$ be 
the set of all smooth functions 
with compact support in $\Omega$.
We put 
$L^2(\mathbb{R}^3)=\mathcal{H}$. 
We denote  
the norm and the inner product of 
$\mathcal{H}$ 
by 
$\|\cdot\|$ and 
$\left\langle \cdot ,\cdot\right\rangle$, 
respectively. 
For $1\le p,q\le \infty$, 
$\|\cdot \|_q$ and $\| \cdot \|_{(p,q)}$ 
denote 
$\|\cdot \|_{L^q(\mathbb{R}^3)}$ 
and 
$\|\cdot \|_{L^p(\mathbb{R};L^q(\mathbb{R}^3))}$, 
respectively. 
We set  
$F(u)=-(V_1\ast |u|^2)u$. 
Let $H$ be 
an unbounded operator on 
${\mathcal H}$ 
defined by 
\[
D(H)=D(-\Delta)=H^2(\mathbb{R}^3), 
\quad 
H=-\Delta+V_0.
\] 
The Kato-Rellich theorem implies that 
$H$ is self-adjoint on $D(H)$ 
(for the detail, see Theorem X.15 in \cite{RS2}). 
Therefore, we see that 
$e^{-itH}:{\mathcal H}\to {\mathcal H}$ 
is a unitary operator. 
That is, we have 
\begin{equation}\label{unitary}
\| e^{-itH}\phi\| =\| \phi\|
\end{equation}
for any $\phi\in\mathcal{H}$. 
Our first result is concerned with 
the direct scattering problem for (\ref{NLS}). 
\begin{theorem}\label{thm:direct}
Assume that 
\begin{equation}\label{RK}
|Q_0|<\mu_0.
\end{equation}
Let $Y_1=L^3(\mathbb{R};L^{18/7}(\mathbb{R}^3))$ 
and $Z_1=C(\mathbb{R};{\mathcal H})\cap Y_1$. 
Then there exists some $\delta>0$ such that 
if $\phi_- \in B(\delta;{\mathcal H})$, 
then  
there uniquely exists $u\in Z_1$ 
such that $u$ is a time-global solution to (\ref{DE})  
and satisfies  
\begin{align}
u(t)&=e^{-itH}\phi_- 
+
\frac{1}{i}\int^{t}_{-\infty}
e^{-i(t-\tau)H}F(u(\tau))d\tau ,\label{IE}\\
\lim_{t\to-\infty}&
\| u(t)-e^{-it\Delta}\phi_-\| =0.\label{AB_1}
\end{align}
Furthermore, 
there exists a unique 
$\phi_+\in {\mathcal H}$
such that 
\begin{align}
\lim_{t\to\infty}&
\| u(t)-e^{-it\Delta}\phi_+\| =0.\label{AB_2}
\end{align}
Therefore, 
the scattering operator for (\ref{NLS})
\[
S_1:B(\delta;{\mathcal H}) 
\ni \phi_- \mapsto \phi_+ \in 
{\mathcal H}
\]
is well-defined. 
\end{theorem}

It is well-known that 
the wave operators 
\[
\Omega_\pm =
s-\lim_{t\to\pm\infty}
e^{itH}e^{it\Delta}:
{\mathcal H}\to P_{ac}({\mathcal H})
\]
and 
the inverse wave operators 
\[
\Omega_\pm^\ast =
s-\lim_{t\to\pm\infty}
e^{-it\Delta}e^{-itH}P_{ac}:
{\mathcal H}\to {\mathcal H}
\]
are well-defined 
(see Theorem XI.30 in \cite{RS3}). 
Here, 
$P_{ac}$ means the projection onto 
the absolutely continuous subspace of $H$. 
Under condition (\ref{RK}), 
then $P_{ac}$ becomes identity 
(see Section \ref{Appendix} below 
and the proof of Theorem XIII.21,(a) in \cite{RS4}). 
We define a mapping $S_{V_0}$ by 
\[
S_{V_0}=\Omega_+^\ast \Omega_- :
{\mathcal H}\to {\mathcal H} .
\]
The operator $S_{V_0}$ is 
the scattering operator for (\ref{DE}) 
with 
$J=-\Delta$ and 
$f(u)=-V_0u$.

Using the method of 
\cite{Watanabe1,Weder97,Weder00-2,Weder01-1,Weder01-2}, 
we see that $S_{V_0}$ 
can be determined 
from the knowledge of $S_1$. 
\begin{theorem}\label{thm2}
(\cite{Watanabe1,Weder97,Weder00-2,Weder01-1,Weder01-2}) 
Assume that (\ref{RK}) holds. 
For any $\phi\in {\mathcal H}\setminus \{ 0\}$, 
we have 
\begin{equation}\label{SV0}
\lim_{\varepsilon\to 0} 
\frac{1}{\varepsilon}{S_1}(\varepsilon \phi) 
=
S_{V_0}(\phi) 
\quad 
\text{in ${\mathcal H}$}.
\end{equation}
\end{theorem} 
Once we have determined $S_{V_0}$, 
we can reconstruct $V_0$, 
$e^{\pm itH}$, $\Omega_{\pm}$, 
$\Omega_{\pm}^\ast$  
by Enss--Weder \cite{Weder}.  
The remaining unknown numbers $Q_1$ and $\mu_1$ 
are determined by the following result:
\begin{theorem}\label{mainthm1}
Assume that (\ref{RK}) holds 
and that 
$\phi\in C^\infty_c(\mathbb{R}^3\setminus \{ 0\})$ 
satisfies $\phi\neq 0$ and 
$(\Delta^2 +1)^{-1}\phi\in 
C^\infty_c(\mathbb{R}^3\setminus \{ 0\})$.
\begin{enumerate}[(i)]
  \item 
We have  
\begin{align}\label{recon-1}
\frac{Q_1}{\mu_1^2}&=
\frac{\displaystyle{
\lim_{\lambda\to\infty}
i\lambda^4
\left\langle
(\Omega_+ S_1 \Omega_-^\ast -id)(\lambda^{-3}\phi_\lambda), 
\phi_\lambda
\right\rangle
}}{\displaystyle{
4\pi \| e^{it\Delta}\phi\|_{(4,4)}^4
}}.
\end{align}
  \item 
Suppose that $Q_1\neq 0$.
Put 
\begin{align*}
b&=
\left|
\frac{Q_1}{\mu_1^2}
\right|^{1/2},\qquad 
H(b)=
-\Delta -bQ_0\frac{\exp(-b\mu_0 r)}{r},\\
\Psi_1(\alpha)&=
\int_{\mathbb{R}}
\left\langle
\frac{\alpha\exp(-\sqrt{|\alpha|}r)}{r}\ast 
\left|
e^{-itH(b)}\phi
\right|^2, 
\left|
e^{-itH(b)}\phi
\right|^2
\right\rangle
dt,
\quad
\alpha\in\mathbb{R},\\
a&=
\lim_{\varepsilon\to 0}
i\varepsilon^{-3}b^{-7}
\left\langle
(\Omega_+ S_1 \Omega_-^\ast -id)(\varepsilon\phi_b), 
\phi_b
\right\rangle, \\
m_0&=
\max\left\{
m\in\mathbb{Z}_{\ge 0};
\Psi_1(m)\le |a|
\right\},\\
q_1&=
\max\left\{
q=0,1;
\Psi_1\left(
m_0+\frac{q}{2}
\right)
\le |a|
\right\},\\
q_{j+1}&=
\max\left\{
q=0,1;
\Psi_1\left(
m_0+\sum_{k=1}^j\frac{q_k}{2^k}+\frac{q}{2^{j+1}}
\right)
\le |a|
\right\},\quad j=1,2,\cdots .
\end{align*} 
Then we have 
\begin{align}\label{recon-2}
Q_1=\mathrm{sign}
\left( \frac{Q_1}{\mu_1^2}\right) 
\left(
m_0+\sum_{j=1}^\infty\frac{q_j}{2^j}
\right) .
\end{align}
\end{enumerate}
\end{theorem}
\begin{rem}
We suppose that $V_0\equiv 0$.
Following the proof of 
Proposition 6 in \cite{Sasaki-Watanabe}, 
we can easily show another 
formula for determining $Q_1/\mu_1^2$ 
\begin{equation}\label{old}
\frac{Q_1}{\mu_1^2}=
\frac{\displaystyle{
\lim_{\lambda\to\infty}
\lambda^{-5}
\lim_{\varepsilon\to 0}
\frac{i}{\varepsilon^3}
\Big\langle 
({S_1} -id)
(\varepsilon \phi_\lambda ), 
\phi_\lambda 
\Big\rangle
}}
{4\pi\| e^{it\Delta} \phi \|_{(4,4)}^4}.
\end{equation} 
Remark that the formula (\ref{recon-1}) is  
simpler than (\ref{old}).
\end{rem}
The contents of this paper 
is as follows: 
In Section \ref{sec:direct}, 
we show Theorem \ref{thm:direct}. 
For this purpose, 
we introduce the $L^p$-$L^q$ estimate 
for solutions to 
the linear Schr\"{o}dinger equation 
\begin{equation}\label{LS}
i\partial_t u +
\Delta u +
Q_0 \frac{\exp(-\mu_0 r)}{r}u
 =0 
\end{equation}
given by Rodnianski--Schlag \cite{Rodnianski-Schlag}.
From the $L^p$-$L^q$ estimate, 
we show that we can treat (\ref{NLS}) 
as the nonlinear Schr\"{o}dinger equation 
\begin{equation}\label{NLS-2}
i\partial_t u +
\Delta u - 
\biggl(
Q_1 \frac{\exp(-\mu_1 r)}{r} \ast |u|^2
\biggr) u =0 
\end{equation} 
wherever we consider only 
the direct scattering problem. 

In Section \ref{SchOp}, 
we consider some properties of 
the Schr\"{o}dinger propagator 
$e^{-itH(\lambda)}$. 
Here, $H(\lambda)$ is a self-adjoint operator 
on $\mathcal{H}$ defined by 
\[
H(\lambda)
=
-\Delta
+
\lambda^2 (V_0)_{\lambda^{-1}} .
\]
We prove that $e^{-itH(\lambda)}$ 
satisfies the following properties:
\begin{itemize}
  \item $\lim_{\lambda\to\infty}e^{-itH(\lambda)}\phi
=e^{it\Delta}\phi$ 
for some $\phi$.
  \item $\left\| e^{-itH(\lambda)} \right\|_6$ 
is bounded with respect to $\lambda>0$.
\end{itemize}
These properties will be used 
to show Theorem \ref{mainthm1}
in Section \ref{sec:Pr. of main 1}.  

In Section \ref{sec:application}, 
we apply Theorem \ref{mainthm1} 
to the inverse scattering problem for 
the semi-relativistic Hartree equation 
\begin{equation}\label{SRH-pre}
\Bigl(
i\partial_t +\sqrt{1-\Delta}
\Bigr)
w +
\biggl(
Q_2 \frac{\exp(-\mu_2 r)}{r} \ast |w|^2
\biggr) w =0
,\quad 
(t,x)\in {\mathbb R}^{1+3}.
\end{equation}
The existence of the scattering operator 
can be shown  
by applying 
the endpoint Strichartz estimate 
for the Klein-Gordon equation 
by Machihara--Nakanishi--Ozawa 
\cite{Machihara-Nakanishi-Ozawa}. 
We prove that 
$\mu_2$ and $Q_2$ 
can be determined via the formulas 
(\ref{recon-3}) and (\ref{recon-4}) below. 
The base of the proof is 
the following limit: 
\begin{equation}\label{pre-NRL}
\lim_{\lambda\to\infty}
\left\| e^{it\lambda^2 -it\lambda\sqrt{\lambda^2-\Delta}}\phi - 
   e^{i\frac{1}{2}\Delta}\phi \right\|_{(4,4)}=0.
\end{equation} 
The functions 
$e^{-it\lambda\sqrt{\lambda^2-\Delta}}\phi$ 
and 
$e^{i\frac{1}{2}\Delta}\phi$ 
are solutions to 
the free semi-relativistic equation 
and 
the free Schr\"{o}dinger equation, 
respectively. 
Thus, 
the limit (\ref{pre-NRL}) is one of 
the non-relativistic limit.  
\section{Direct Problem}\label{sec:direct}
In this section, 
we first prepare the Key properties 
to show Theorems 
\ref{thm:direct} and \ref{mainthm1}. 

For a measurable function 
$V:\mathbb{R}^3\to\mathbb{C}$, 
we set 
\begin{align*}
\| V\|_{R} &=
\sqrt{
\int_{\mathbb{R}^{3+3}}
\frac{|V(x)V(y)|}{|x-y|^2}d(x,y)},\\
\| V\|_{\mathcal K} &=
\sup_{x\in\mathbb{R}^3}
\int_{\mathbb{R}^3}
\frac{|V(y)|}{|x-y|}dy .
\end{align*}
The norm  
$\|\cdot\|_R$ is called 
the Rollnik norm.
Remark that 
the condition (\ref{RK}) is equivalent to 
\begin{align}\label{RK-e}
\frac{|Q_0|}{\mu_0}<
4\pi \min
\Biggl\{
\biggl\| \dfrac{e^{-r}}{r} \biggl\|_R^{-1}, 
\biggl\| \dfrac{e^{-r}}{r} \biggl\|_{\mathcal K}^{-1}
\Biggr\}.
\end{align}
For the detail, see Section \ref{Appendix} below.

Under some suitable condition of $V$, 
we obtain the following 
the time-decay estimate of $e^{-it(-\Delta+V)}$: 
\begin{prop}\label{prop:decay}
(\cite{Rodnianski-Schlag}) 
Suppose that a measurable function 
$V:\mathbb{R}^3\to \mathbb{R}$ 
satisfies 
\begin{equation}\label{RK2}
\max\Bigl\{
\| V \|_R , 
\| V \|_{\mathcal K}
\Bigr\} 
<4\pi .
\end{equation}
Then we have 
\begin{equation}\label{decay}
\left\| e^{-it(-\Delta+V)}\phi \right\|_\infty \le C
|t|^{-\frac{3}{2}}\| \phi\|_1
\end{equation}
for all 
$\phi\in L^1(\mathbb{R}^3)$ and $t\neq 0$.
\end{prop} 
Assume that $V$ satisfies (\ref{RK2}). 
By (\ref{unitary}) and (\ref{decay}), 
it follows from 
the Riesz--Thorin interpolation theorem that 
we obtain the $L^p$-$L^q$ estimate 
\begin{equation}\label{LpLq0}
\| e^{-it(-\Delta+V)}\phi\|_p \le C
|t|^{-\frac{3}{2}}\| \phi\|_q
\end{equation}
for all 
$1\le q\le 2\le p\le \infty$ with 
$1/p+1/p=1$,  
$\phi\in L^q(\mathbb{R}^3)$ and $t\neq 0$. 
It is shown by Ginibre--Velo 
\cite{Ginibre-Velo85-2} that 
via $T^\ast T$ argument, 
(\ref{LpLq0}) gives rise to 
the class of the Strichartz type estimates.
\begin{prop}\label{prop:st}
Assume that $V$ satisfies (\ref{RK2}). 
Let $q_j>2$, 
$2<\beta_j<6$, 
$j=1,2,3$. 
If $2/q_j=3/2-3/\beta_j$, 
$j=1,2,3$, then we have 
\begin{align}
\bigl\| 
e^{\pm it(-\Delta +V)}\phi
\bigl\|_{(q_1,\beta_1)}
&\le C
\| \phi\| \label{st1},\\
\biggl\|
\int^{t}_{-\infty}
e^{\pm i(t-\tau)(-\Delta +V)}f(\tau)d\tau
\biggl\|_{(q_2,\beta_2)}
&\le C
\| f\|_{(q_3^\prime,\beta_3^\prime)}. \label{st2}
\end{align}
Here, $q_3^\prime$ and $\beta_3^\prime$ denote 
the H\"{o}lder conjugate of $q_3$ and $\beta_3$, 
respectively. 
\end{prop}
We are ready to show Theorem \ref{thm:direct}. 
\begin{proof}[Proof of Theorem \ref{thm:direct}]
Put 
\[
F(u)=\biggl(
Q_1 \frac{\exp(-\mu_1 r)}{r} \ast |u|^2
\biggr) u
\]
and 
\[
\Psi[u](t)=
e^{-itH}\phi_- 
+
\frac{1}{i}\int^{t}_{-\infty}
e^{-i(t-\tau)H}F(u(\tau))d\tau .
\]
By (\ref{RK-e}), we have 
\[
\max\Bigl\{
\| V_0 \|_R , 
\| V_0 \|_{\mathcal K}
\Bigr\} 
<4\pi .
\] 
Therefore, from (\ref{st1}) and (\ref{st2}), 
we obtain 
\[
\| \Psi[u]\|_{Z_1} \le C 
\Bigl(
\| \phi_-\| +
\| F(u)\|_{(1,2)}
\Bigr) .
\]
Since $V_0\in L^{3/2}(\mathbb{R}^3)$, 
we see that 
\[
\| F(u)(t)\| \le 
\bigl\| V_0 \ast |u(t)|^2 \bigl\| \| u(t)\|_{18/7} 
\le 
\| V_0\|_{3/2} \| u(t)\|_{18/7}^3,
\]
where we have used 
the H\"{o}lder--Young inequality 
in the second inequality. 
Hence we see that 
\[
\| \Psi[u]\|_{Z_1} 
\le C 
\Bigl(
\| \phi_-\| +
\| u\|_{Z_1}^3
\Bigr) .
\]
Similarly, we  obtain 
\[
\| \Psi[u]-\Psi[\tilde{u}]\|_{Z_1} 
\le C 
\| u-\tilde{u}\|_{Z_1}
\Bigl(
\| u\|_{Z_1} + \| \tilde{u}\|_{Z_1}
\Bigr)^2 .
\]
It is clear that 
$\Psi[u]\in C(\mathbb{R};{\mathcal H})$. 
Therefore, we see that 
there uniquely exists 
$u\in Z_1$ such that 
$\Psi[u]=u$ for sufficiently small $\delta>0$. 
We can immediately find that 
the fixed point $u$ solves the equation (\ref{NLS}).
Furthermore, we obtain  
\begin{align}
\| u\|_{Z_1} &\le C\| \phi_-\| ,\label{K1}\\
\| u-e^{-itH}\phi_-\|_{Z_1} 
&\le C \| \phi_- \|^3 .\label{K2}
\end{align}
It follows 
from $u\in L^3(\mathbb{R};L^{18/7}(\mathbb{R}^n))$
that 
\begin{align*}
\| u(t)-e^{-itH}\phi_-\|
&\le 
\int^{\infty}_{t}
\| F(u)(t)\| dt \\
&\le 
\int^{t}_{-\infty}
\| V_0\|_{3/2} 
\| u(t)\|_{18/7}^3 dt \\
&\to 0 
\quad\text{as $t\to \infty$}.
\end{align*} 
Furthermore, 
if we put 
\begin{equation}\label{SF}
S_F(\phi_-)=\phi_- +
\frac{1}{i}\int_{\mathbb{R}} 
e^{itH}(F(u(t)))dt ,
\end{equation}
we have 
\begin{align*}
\lim_{t\to +\infty}
\| u(t)-e^{-itH}S_F(\phi_-) \|
=0.
\end{align*}

From $\| \Omega_-(\phi_-)\|= \| \phi_-\|$,  
there exists some $\delta_0>0$ such that 
if $\phi_-\in B(\delta_0,{\mathcal H})$, 
then there uniquely exists 
$v\in Z_1$ satisfying $\Psi[v]=v$ and 
\begin{align}\label{AB2-1}
\lim_{t\to -\infty}
\| v(t)-e^{-itH}\Omega_-(\phi_-)\| 
=0 .
\end{align}
Moreover, we have 
\begin{align}\label{AB2-2}
\lim_{t\to \infty}
\| v(t)- e^{-itH} S_F \Omega_-(\phi_-)\| 
=0 .
\end{align}
By (\ref{unitary}), 
we see from (\ref{AB2-1}) and (\ref{AB2-2}) that  
\begin{align*}
\| v(t)&-e^{it\Delta}\phi_-\| \\
&\le 
\| v(t)-e^{-itH}\Omega_-(\phi_-)\| +
\| e^{-itH}\Omega_-(\phi_-) - e^{it\Delta}\phi_- \| \\
&= 
\| v(t)-e^{-itH}\Omega_-(\phi_-)\| +
\| \Omega_-(\phi_-) - e^{itH}e^{it\Delta}\phi_- \|  \\
&\to 0 
\quad\text{as $t\to -\infty$}
\end{align*} 
and 
\begin{align*}
\| &v(t)-e^{it\Delta}
\Omega_+^\ast S_F\Omega_-(\phi_-)\| \\
&\le 
\| v(t)-e^{-itH}S_F\Omega_-(\phi_-)\| +
\| e^{-itH}S_F\Omega_-(\phi_-) - 
e^{it\Delta}\Omega_+^\ast S_F\Omega_-(\phi_-) \| \\
&= 
\| v(t)-e^{-itH}S_F\Omega_-(\phi_-)\| +
\| S_F\Omega_-(\phi_-) - 
e^{itH}e^{it\Delta}\Omega_+^\ast 
S_F \Omega_-(\phi_-) \|  \\
&\to 0 
\quad\text{as $t\to +\infty$},
\end{align*}  
respectively. 
Thus, (\ref{AB_1}) and (\ref{AB_2}) hold  
if we define ${S_1}=\Omega_+^\ast S_F \Omega_-$. 
This completes the proof. 
\end{proof}
\section{Schr\"{o}dinger Propagator}\label{SchOp}
For $\lambda>0$ and $y\in\mathbb{R}^3$, 
let $H(\lambda,y)$ be a linear operator 
on $\mathcal{H}$ defined by 
\[
D(H(\lambda,y))=D(-\Delta),\quad
H(\lambda,y)=-\Delta+\lambda^2 V_0(\lambda x-y).
\]
Furthermore, we put 
\[
H(\lambda)=H(\lambda,0). 
\]
The operator $H(\lambda,y)$ becomes 
a self-adjoint operator on $\mathcal{H}$.
In this section, 
we list some properties of 
the Schr\"{o}dinger propagator $e^{-itH(\lambda,y)}$.
The properties are useful 
to prove Theorem \ref{mainthm1}.
\begin{prop}\label{prop:3-1}
Let $\lambda>0$ and $\phi\in\mathcal{H}$. 
\begin{enumerate}[(i)]
  \item 
\begin{align}\label{H(l)-1}
\left(
e^{-itH}\phi_\lambda
\right) (x)
=
\left(
e^{-i\lambda^{-2}tH(\lambda)}\phi
\right)(\lambda^{-1}x) .
\end{align}
  \item 
\begin{align}\label{H(l)-2}
\left(
e^{-itH(\lambda)}\phi
\right) (x-\lambda^{-1}y)
=
\left(
e^{-itH(\lambda,y)}\tau_{\lambda^{-1}y}\phi
\right)(x) ,
\end{align}
where $\tau_z\phi(x):=\phi(x-z)$, $z\in\mathbb{R}^3$.
\end{enumerate}
\end{prop}
\begin{proof}
Put 
$u(t,x)=e^{-itH(\lambda)}\phi(x)$. 
Then we see that 
\begin{align*}
(i\partial_t+\Delta -V)u(\lambda^{-2}t,\lambda^{-1}x)
=
\lambda^{-2}
\left(
i\partial_t u+\Delta u-\lambda^2V(\lambda\cdot)u
\right) (\lambda^{-2}t,\lambda^{-1}x)=0 
\end{align*} 
and 
\[
u(0,\lambda^{-1}x)=\phi_\lambda(x).
\]
Therefore, we obtain 
\[
e^{-itH}\phi_\lambda(x)
=
u(\lambda^{-2}t,\lambda^{-1}x)
=
e^{-it\lambda^{-2}H(\lambda)}\phi(\lambda^{-1}x).
\]

Furthermore, we have 
\begin{align*}
&\left(
i\partial_t+\Delta -\lambda^2V(\lambda(x-\lambda^{-1}y))
\right)
u(t,x-\lambda^{-1}y)\\
&\qquad\qquad\qquad =
\left(
i\partial_t u+\Delta u-\lambda^2V(\lambda\cdot)u
\right) (t,x-\lambda^{-1}y)=0 
\end{align*}
and 
\[
u(0,x-\lambda^{-1}y)=\tau_{\lambda^{-1}y}\phi(x).
\]
Therefore, we obtain 
\begin{align*}
e^{-itH(\lambda,y)}\tau_{\lambda^{-1}y}\phi(x)
=
u(t,x-\lambda^{-1}y)
=
e^{-itH(\lambda)}\phi(x-\lambda^{-1}y).
\end{align*} 
\end{proof}
\begin{prop}\label{prop:conv-expitH}
If $\phi\in C^\infty_c(\mathbb{R}^3\setminus\{ 0\})$, 
then we have 
\begin{align}\label{conv-expitH}
\lim_{\lambda\to\infty}
\left\|
\left(
e^{-itH(\lambda,y)}-e^{it\Delta}
\right)
\left(
(-\Delta+i)(-\Delta-i)\phi
\right)
\right\| =0.
\end{align}
\end{prop}
\begin{proof}
The proof is essentially similar to that of 
Theorem VIII.20 in \cite{RS1}.
However, for the sake of completeness, 
we give here the proof dividing two steps.\\
(Step I.)\
For $\alpha\in\mathbb{R}$ and $m\in\mathbb{Z}_{\ge 0}$, 
let $f(\alpha)=e^{-it\alpha}$ and  
$g_m(\alpha)=e^{-\alpha^2/m^2}$.  
Let $\psi\in \mathcal{H}$ and $\varepsilon>0$. 
Then there exists some $m_0\in\mathbb{Z}_{\ge 0}$ 
such that 
\begin{align}\label{m-0}
\| g_m(-\Delta) \psi -\psi \| \le \varepsilon
\end{align}
for any $m\ge m_0$.
Henceforth, we assume that $m= m_0$.
Since $g_m$ is a continuous function 
vanishing at infinity,  
it follows from the Stone-Weierstrass theorem 
(see, e.g., \cite{RS1}) that 
there exists some two-parameter polynomial 
$P(\alpha,\beta)$ such that 
\[
\sup_{\alpha\in\mathbb{R}}
\left|
g_m(\alpha)
-
P\left(
(\alpha+i)^{-1},(\alpha-i)^{-1}
\right)
\right|
\le
\varepsilon .
\]
Therefore, 
for any self-adjoint operator $A$, 
we have 
\begin{align}\label{gm(A)}
\left\|
g_m(A)
-
P\left(
(A+i)^{-1},(A-i)^{-1}
\right)
\right\|
\le
\varepsilon .
\end{align}
Thus, we obtain 
\begin{align}
&\| 
(g_m(H(\lambda,y))-g_m(-\Delta))\psi
\| \nonumber\\
&\le 
\left\|
g_m(H(\lambda,y))
-
P\left(
(H(\lambda,y)+i)^{-1},(H(\lambda,y)-i)^{-1}
\right)
\right\|
\| \psi\| \nonumber\\
&\quad +
\left\|
g_m(-\Delta)
-
P\left(
(-\Delta+i)^{-1},(-\Delta-i)^{-1}
\right)
\right\|
\| \psi\| \nonumber\\
&\quad +
\left\|
P\left(
(H(\lambda,y)+i)^{-1},(H(\lambda,y)-i)^{-1}
\right) \psi
-
P\left(
(-\Delta+i)^{-1},(-\Delta-i)^{-1}
\right) \psi
\right\| \nonumber\\
&\le 
2\varepsilon \| \psi\| 
+
\left\|
P\left(
(H(\lambda,y)+i)^{-1},(H(\lambda,y)-i)^{-1}
\right) \psi
-
P\left(
(-\Delta+i)^{-1},(-\Delta-i)^{-1}
\right) \psi
\right\| .
\label{polypoly}
\end{align}
Here, we have used  
the property (\ref{gm(A)}) in the last inequality.
Since it follows that 
\[
P(\alpha_1,\beta_1)
-
P(\alpha_2,\beta_2)
=
\sum_{k,l}C_{k,l}
\left\{
(\alpha_1-\alpha_2)
\widetilde{P}^{k-1}(\alpha_1,\alpha_2)\beta_1^l
+
(\beta_1-\beta_2)
\widetilde{P}^{l-1}(\beta_1,\beta_2)\alpha_2^k
\right\}
\]
for some two-parameter polynomial 
$\widetilde{P}^{k}$,
$k=-1,0,1,2,\cdots$, 
we obtain 
\begin{align}
&\left\|
 P\left(
(H(\lambda,y)+i)^{-1},(H(\lambda,y)-i)^{-1}
\right) \psi
-
P\left(
(-\Delta+i)^{-1},(-\Delta-i)^{-1}
\right) \psi
\right\| \nonumber\\
&\ \le 
C\left\{
\left\|
\left(
(H(\lambda,y)+i)^{-1}-(-\Delta+i)^{-1}
\right) \psi
\right\|
+
\left\|
\left(
(H(\lambda,y)-i)^{-1}-(-\Delta-i)^{-1}
\right) \psi
\right\|
\right\} .\label{polypoly2}
\end{align} 
Henceforth, we put 
$\psi=(-\Delta+i)(-\Delta-i)\phi$, 
$\phi\in C^\infty_c(\mathbb{R}^3\setminus\{ 0\})$. 
Then we see that 
\begin{align*}
&\left\|
\left(
(H(\lambda,y)\pm i)^{-1}-(-\Delta\pm i)^{-1}
\right) \psi
\right\| \\
&\quad\le 
\left\|
\left(
(H(\lambda,y)\pm i)^{-1}-(-\Delta\pm i)^{-1}
\right) (-\Delta+i)(-\Delta-i)\phi
\right\| \\
&\quad =
\left\|
(H(\lambda,y)\pm i)^{-1}
\lambda^2V(\cdot -\lambda^{-1}y)
(-\Delta\mp i)\phi
\right\| \\
&\quad\le 
\left\|
\lambda^2V(\cdot -\lambda^{-1}y)
(-\Delta\mp i)\phi
\right\| .
\end{align*} 
We set 
$\eta=\mathrm{dist}(\mathrm{supp}\phi,0)$.
If $\lambda>0$ is sufficiently large, 
then we have 
\begin{align}
\left\|
\left(
(H(\lambda,y)\pm i)^{-1}-(-\Delta\pm i)^{-1}
\right) \psi
\right\| 
\le C
|Q_0|\lambda^2
\frac{\exp(-\mu_0|\lambda\eta-y|)}{|\lambda\eta-y|}
\| (-\Delta\mp i)\phi\| .\label{etalambda}
\end{align} 
We see from (\ref{polypoly})--(\ref{etalambda})
that 
\begin{align}\label{gm-conv}
\lim_{\lambda\to\infty}\| g_m(H(\lambda,y))\psi - 
g_m(-\Delta)\psi\| 
=0.
\end{align} 
(Step II.)\ 
Since $f g_m$ is a continuous function 
vanishing at infinity, 
it follows from 
the same argument of the proof of (\ref{polypoly}) 
that 
\begin{align}
\| 
&(f(H(\lambda,y))-f(-\Delta))\psi
\| \nonumber\\
&\le 
\left\| 
(fg_m)(H(\lambda,y))\psi 
-
f(H(\lambda,y))\psi
\right\| 
+
\left\| 
(fg_m)(-\Delta)\psi 
-
f(-\Delta)\psi
\right\| \nonumber\\
&\quad +
\left\| 
(fg_m)(H(\lambda,y))\psi 
-
(fg_m)(-\Delta)\psi
\right\| \nonumber\\
&\le 
\| f(H(\lambda,y))\| 
\| g_m(H(\lambda,y))\psi-\psi\| 
+
\| f(-\Delta)\| 
\| g_m(-\Delta)\psi-\psi\| 
+
2\varepsilon \| \psi\| \nonumber\\
&\quad +
\left\|
\acute{P}\left(
(H(\lambda,y)+i)^{-1},(H(\lambda,y)-i)^{-1}
\right) \psi
-
\acute{P}\left(
(-\Delta+i)^{-1},(-\Delta-i)^{-1}
\right) \psi
\right\|
\nonumber\\
&\le 
\| g_m(H(\lambda,y))\psi-g_m(-\Delta)\psi\|
+
2\| g_m(-\Delta)\psi-\psi\| 
+
2\varepsilon \| \psi\| \nonumber\\
&\quad +
\left\|
\acute{P}\left(
(H(\lambda,y)+i)^{-1},(H(\lambda,y)-i)^{-1}
\right) \psi
-
\acute{P}\left(
(-\Delta+i)^{-1},(-\Delta-i)^{-1}
\right) \psi
\right\| 
\end{align}
for some two-parameter polynomial $\acute{P}$.
Here, we have used the unitarity of $f(A)$
in the last inequality.
By (\ref{m-0}) and (\ref{polypoly2})--(\ref{gm-conv}),
(\ref{conv-expitH}) holds. 
\end{proof}
\begin{prop}\label{prop:bdd-expitH}
For any $\lambda>0$ and for any
$\phi\in C^\infty_c(\mathbb{R}^3\setminus\{ 0\})$, 
we have 
\begin{align}\label{bdd-eitH}
\left\|
e^{-itH(\lambda)}\phi
\right\|_6 
\le C(\phi) 
\langle t\rangle^{-1}.
\end{align}
Here, the constant $C(\phi)$ is 
independent of $\lambda$ and $t$.
\end{prop}
\begin{proof}
Let $C_b$ be the best constant 
for the embedding 
$\dot{H}^1(\mathbb{R}^3)\hookrightarrow L^6(\mathbb{R}^3)$: 
\[
C_b=
\sup_{\psi\in \dot{H}^1\setminus\{ 0\}}
\frac{\| \psi\|_6}{\| \nabla \psi\|}.
\]
Then we obtain 
\begin{align*}
\left\|
e^{-itH(\lambda)}\phi
\right\|_6^2  
&\le 
C_b^2 \left\|
\nabla e^{-itH(\lambda)}\phi
\right\|^2 \\
&\le 
C_b^2 
\left\langle
\nabla e^{-itH(\lambda)}\phi,
\nabla e^{-itH(\lambda)}\phi
\right\rangle \\
&\le 
C_b^2 
\left\langle
(-\Delta)e^{-itH(\lambda)}\phi,
e^{-itH(\lambda)}\phi
\right\rangle \\
&\le 
C_b^2 
\left\langle
H(\lambda) e^{-itH(\lambda)}\phi,
e^{-itH(\lambda)}\phi
\right\rangle 
-
C_b^2 
\left\langle
\lambda^2{(V_0)}_{\lambda^{-1}}e^{-itH(\lambda)}\phi,
e^{-itH(\lambda)}\phi
\right\rangle \\
&\le 
C_b^2 
\left\langle
H(\lambda) \phi,
\phi
\right\rangle 
+
C_b^2 
\| \lambda^2{(V_0)}_{\lambda^{-1}}\|_{3/2}
\left\|
e^{-itH(\lambda)}\phi
\right\|_6^2 \\
&\le 
C_b^2 
\| \nabla\phi\|^2
+
C_b^2 
\left|
\left\langle
\lambda^2{(V_0)}_{\lambda^{-1}} \phi,
\phi
\right\rangle 
\right| 
+
C_b^2 
\| V_0\|_{3/2}
\left\|
e^{-itH(\lambda)}\phi
\right\|_6^2 .
\end{align*} 
It follows that 
\begin{align*}
\left|
\left\langle
\lambda^2{(V_0)}_{\lambda^{-1}} \phi,
\phi
\right\rangle 
\right| 
\le 
|Q_0|\lambda\frac{\exp(-\mu_0\eta\lambda)}{\eta}\|\phi\|^2 
\le
\frac{|Q_0|\|\phi\|^2}{e\mu_0\eta^2},
\end{align*} 
where 
$\eta=\mathrm{dist}(\mathrm{supp}\phi,0)$.
Since 
\begin{align}\label{C-b3/2}
C_b^2
\left\|
\frac{e^{-r}}{r}
\right\|_{3/2}
<1
\end{align}
(for the proof, see Section \ref{Appendix} below), 
we see that 
\begin{align}\label{L6Bdd}
\left\|
e^{-itH(\lambda)}\phi
\right\|_6
\le 
\sqrt{
\frac{\displaystyle{
\|\nabla\phi\|^2
+
\frac{|Q_0|}{e\mu_0\eta^2}\|\phi\|^2
}}{\displaystyle{
C_b^{-2}-\frac{|Q_0|}{\mu_0}
\left\|
\frac{e^{-r}}{r}
\right\|_{3/2}
}}
}.
\end{align}

Furthermore, 
by (\ref{LpLq0}) and (\ref{H(l)-1}), 
we have 
\begin{align}
\left\|
e^{-itH(\lambda)}\phi
\right\|_6
&=
\left\|
\left(
e^{-it\lambda^2 H(\lambda)}\phi
\right)_{\lambda^{-1}}
\right\|_6\nonumber\\
&=
\lambda^{-1/2}
\left\|
e^{-it\lambda^2 H(\lambda)}\phi
\right\|_6\nonumber\\
&\le C
\lambda^{-1/2}
|t\lambda^2|^{-3(1/2-1/6)}
\| \phi_\lambda\|_{6/5}\nonumber\\
&= C
\lambda^{-1/2-2+5/2}t^{-1}\|\phi\|_{6/5}\nonumber\\
&= C
t^{-1}\|\phi\|_{6/5}.\label{L6decay}
\end{align} 
From (\ref{L6Bdd}) and (\ref{L6decay}), 
we have (\ref{bdd-eitH}).
\end{proof}
\section{Proof of Theorem \ref{mainthm1}}
\label{sec:Pr. of main 1}
As we mention in Section \ref{sec:intro}, 
we can reconstruct $V_0$ 
from the knowledge of scattering states 
$(\phi_-,{S_1}(\phi_-))$. 
In this section, we give the proof of 
Theorem \ref{mainthm1},  
which enables us to see the exact form of $V_1$.

Set 
\[
K[\phi]=
\lim_{\varepsilon\to 0}
\frac{i}{\varepsilon^3}
\Big\langle 
(\Omega_+ {S_1} \Omega_-^\ast - id)
(\varepsilon \phi), \phi 
\Big\rangle .
\]
Using (\ref{K1}) and (\ref{K2}), 
we have the following property:
\begin{prop}\label{prop:SAL}
(Strauss \cite{Strauss 1973}) 
Assume that (\ref{RK}) holds. 
Then we have for all $\phi\in {\mathcal H}$, 
\begin{equation}\label{SAL}
K[\phi]=\int_{\mathbb{R}} 
\Big\langle 
F(e^{-itH}\phi),\phi
\Big\rangle dt .
\end{equation}
\end{prop}
We are now ready to state 
the proof of Theorem \ref{mainthm1}.
\begin{proof}[Proof of Theorem \ref{mainthm1}]
Remark that we obtain  
For any $\phi\in {\mathcal H}\setminus \{ 0\}$ 
and any $\lambda> \| \phi\|^{3/2}\delta^{-1}$, 
$(\Omega_+ {S_1} \Omega_-^\ast -id)(\lambda^{-3}\phi_\lambda)$ is 
well-defined because we have 
\[
\| \lambda^{-3}\phi_\lambda\| \le 
\lambda^{-3/2}\| \phi\| . 
\]  

Let $u_\lambda$ be 
the time-global solution to (\ref{IE}) 
satisfying $\phi_-=\lambda^{-3}\phi_\lambda$. 
Put 
\begin{align*}
u^0_\lambda = 
e^{it\Delta}(\lambda^{-3}\phi_\lambda) ,
\quad 
\widetilde{u^0_\lambda}= 
e^{it\Delta}(\phi_\lambda),
\quad 
u^1_\lambda = 
u_\lambda -u^0_\lambda .
\end{align*} 
Then we obtain 
\[
i\lambda^4
\Big\langle 
(\Omega_+ {S_1} \Omega_-^\ast - id)
(\lambda^{-3}\phi_\lambda), \phi_\lambda 
\Big\rangle 
=
(I)_\lambda +
(II)_\lambda^1+(II)_\lambda^2+(II)_\lambda^3,
\]
where 
\begin{align*}
(I)_\lambda &= 
\lambda^4 \int_{\mathbb{R}} 
\Big\langle 
\bigl(
V_1\ast u^0_\lambda \overline{u^0_\lambda}
\bigr) u^0_\lambda, 
\widetilde{u^0_\lambda}
\Big\rangle dt, \\
(II)_\lambda^1 &= 
\lambda^4 \int_{\mathbb{R}} 
\Big\langle 
\bigl(
V_1\ast u^1_\lambda \overline{u^0_\lambda}
\bigr) u^0_\lambda, 
\widetilde{u^0_\lambda}
\Big\rangle dt, \\
(II)_\lambda^2 &= 
\lambda^4 \int_{\mathbb{R}} 
\Big\langle 
\bigl(
V_1\ast u_\lambda \overline{u^1_\lambda}
\bigr) u^0_\lambda, 
\widetilde{u^0_\lambda}
\Big\rangle dt, \\
(II)_\lambda^3 &= 
\lambda^4 \int_{\mathbb{R}} 
\Big\langle 
\bigl(
V_1\ast u_\lambda \overline{u_\lambda}
\bigr) u^1_\lambda, 
\widetilde{u^0_\lambda}
\Big\rangle dt .
\end{align*} 
\pagebreak
Following the proof of 
Proposition \ref{SAL}, 
we see that for $j=1,2,3$, 
\[
\bigl|
(II)_\lambda^j 
\bigl| 
\le C 
\lambda^{-7/2}\| \phi\|^6 
\to 0 
\quad\text{as $\lambda\to\infty$}.
\]
Following Proposition \ref{prop:3-1}, 
we obtain 
\begin{align*}
(I)_\lambda &= 
\lambda^{-5}\int_{\mathbb{R}^{7}}
Q_1\frac{\exp(-\mu_1 |y|)}{|y|}
\bigl| e^{-i\lambda^{-2}tH(\lambda)}\phi
(\lambda^{-1}x-\lambda^{-1}y) \bigl|^2 
\bigl| e^{-i\lambda^{-2}tH(\lambda)}\phi
(\lambda^{-1}x) \bigl|^2 d(t,x,y)\\
&=
\int_{\mathbb{R}^{7}}
Q_1\frac{\exp(-\mu_1 |y|)}{|y|}
\bigl| 
e^{-itH(\lambda,y)}\tau_{\lambda^{-1}y}\phi (x) 
\bigl|^2 
\bigl| e^{-itH(\lambda)}\phi(x) \bigl|^2 d(t,x,y) \\
&=
\int_{\mathbb{R}^3}
Q_1\frac{\exp(-\mu_1 |y|)}{|y|}
\Phi(\lambda,y)dy ,
\end{align*}
where 
\[
\Phi(\lambda,y)
=
\int_{\mathbb{R}^{1+3}}
\bigl| 
e^{-itH(\lambda,y)}\tau_{\lambda^{-1}y}\phi (x) 
\bigl|^2 
\bigl| e^{-itH(\lambda)}\phi(x) \bigl|^2 d(t,x) .
\]
 
For the function $\Phi$, 
we have the following property:
\begin{lem}\label{lem-Phi}
Assume that 
$\phi\in C^\infty_c(\mathbb{R}^3\setminus \{ 0\})$ 
satisfies $\phi\neq 0$ and 
$(\Delta^2 +1)^{-1}\phi\in 
C^\infty_c(\mathbb{R}^3\setminus \{ 0\})$.
\begin{enumerate}[(i)]
  \item For any $y\in\mathbb{R}^3$, we have 
\[
\lim_{\lambda\to\infty}\Phi(\lambda,y)=
\| e^{it\Delta}\phi \|_{(4,4)}^4 .
\]
  \item 
For any $\lambda>0$ and $y\in\mathbb{R}^3$, 
we have  
\[ 
| \Phi(\lambda,y) |\le C(\phi) ,
\]  
where the constant $C(\phi)$ is 
independent of $\lambda$ and $y$.
\end{enumerate}
\end{lem}
\begin{proof}[Proof of Lemma \ref{lem-Phi}]
It follows from the H\"{o}lder inequality that 
\begin{align*}
&\left|
\Phi(\lambda,y) -\| e^{it\Delta}\phi \|_{(4,4)}^4
\right| \\
&\quad\le 
\int_{\mathbb{R}^{1+3}}
\left(
\bigl| 
e^{-itH(\lambda,y)}\tau_{\lambda^{-1}y}\phi (x) 
\bigl|^2 
\bigl| e^{-itH(\lambda)}\phi(x) \bigl|^2 
-
|e^{it\Delta}\phi(x) |^4
\right)
d(t,x) \\
&\quad\le
\int_{\mathbb{R}^{1+3}}
\bigl| 
e^{-itH(\lambda,y)}\tau_{\lambda^{-1}y}\phi (x) 
-
e^{it\Delta}\phi(x)
\bigl|
\bigl| 
e^{-itH(\lambda,y)}\tau_{\lambda^{-1}y}\phi (x) 
\bigl|
\bigl| e^{-itH(\lambda)}\phi(x) \bigl|^2 
d(t,x) \\
&\qquad +
\int_{\mathbb{R}^{1+3}}
\bigl| 
e^{it\Delta}\phi (x) 
\bigl|
\bigl| 
e^{-itH(\lambda,y)}\tau_{\lambda^{-1}y}\phi (x) 
-
e^{it\Delta}\phi(x)
\bigl|
\bigl| e^{-itH(\lambda)}\phi(x) \bigl|^2 
d(t,x) \\
&\qquad +
\int_{\mathbb{R}^{1+3}}
\bigl| 
e^{it\Delta}\phi(x)
\bigl|^2
\bigl| 
e^{-itH(\lambda)}\phi (x) 
-
e^{it\Delta}\phi(x)
\bigl|
\bigl| e^{-itH(\lambda)}\phi(x) \bigl|
d(t,x) \\
&\qquad +
\int_{\mathbb{R}^{1+3}}
\bigl| 
e^{it\Delta}\phi(x)
\bigl|^3
\bigl| 
e^{-itH(\lambda)}\phi (x) 
-
e^{it\Delta}\phi(x)
\bigl|
d(t,x) \\
&\quad\le
\int_{\mathbb{R}}
\left\|
e^{-itH(\lambda,y)}\tau_{\lambda^{-1}y}\phi  
-
e^{it\Delta}\phi
\right\| 
\left\|
e^{-itH(\lambda)}\phi  
\right\|_6^3 
dt \\
&\qquad +
\int_{\mathbb{R}}
\left\|
e^{it\Delta}\phi  
\right\|_6
\left\|
e^{-itH(\lambda,y)}\tau_{\lambda^{-1}y}\phi  
-
e^{it\Delta}\phi
\right\| 
\left\|
e^{-itH(\lambda)}\phi  
\right\|_6^2 
dt \\
&\qquad +
\int_{\mathbb{R}}
\left\|
e^{it\Delta}\phi  
\right\|_6^2
\left\|
e^{-itH(\lambda)}\phi  
-
e^{it\Delta}\phi
\right\| 
\left\|
e^{-itH(\lambda)}\phi  
\right\|_6
dt \\
&\qquad +
\int_{\mathbb{R}}
\left\|
e^{it\Delta}\phi  
\right\|_6^3
\left\|
e^{-itH(\lambda)}\phi  
-
e^{it\Delta}\phi
\right\| 
dt 
\end{align*} 
where we have used the equality 
\[
\left\| 
e^{-itH(\lambda,y)}\tau_{\lambda^{-1}y}\phi
\right\|_6
=
\left\| 
e^{-itH(\lambda)}\phi 
\right\|_6,
\]
which is given by (\ref{H(l)-2}), 
in the last inequality.
We can easily see that 
\[
\left\|
e^{it\Delta}\phi  
\right\|_6 
\le 
\langle t \rangle^{-1}
\left(
\| \nabla\phi\| 
+
\| \phi\|_{6/5}
\right) .
\]
Therefore, 
by Propositions 
\ref{prop:conv-expitH} and \ref{prop:bdd-expitH} 
and 
by applying the Lebesgue dominated theorem 
with respect to the variable $t$, 
we obtain (i). 
Similarly, we have (ii). 
\end{proof}

Let us go back to the proof of Theorem \ref{mainthm1}.
Henceforth, we suppose that 
$\phi\in C^\infty_c(\mathbb{R}^3\setminus \{ 0\})$ 
satisfies $\phi\neq 0$ and 
$(\Delta^2 +1)^{-1}\phi\in 
C^\infty_c(\mathbb{R}^3\setminus \{ 0\})$.
Using the above Lemma \ref{lem-Phi} 
and Prop \ref{norm-of-Yukawa},(iii),  
we see from the Lebesgue dominated theorem 
with respect to the variable $y$
that 
\begin{align*}
\lim_{\lambda \to\infty}(I)_\lambda 
&= 
\| e^{it\Delta}\phi \|_{(4,4)}^4\int_{\mathbb{R}^3}
Q_1\frac{\exp(-\mu_1 |y|)}{|y|}dy \\
&=
4\pi
\frac{Q_1}{\mu_1^2}
\| e^{it\Delta}\phi\|_{(4,4)}^4,
\end{align*} 
which implies (\ref{recon-1}). \\

We next show (\ref{recon-2}). 
Suppose that $Q_1\neq 0$. 
Recall the definition of 
$b$, $\Psi$, $m_0$ and $q_j$, $j=1,2,\cdots$. 
It follows from Propositions \ref{prop:SAL} and \ref{prop:3-1}, 
that 
\begin{align*}
a
=
b^{-7}K[\phi_b]
&=
b^{-7}Q_1
\int_{\mathbb{R}}
\left\langle
\frac{\exp(-\mu_1 r)}{r}\ast 
\left|
e^{-itH}\phi_b
\right|^2, 
\left|
e^{-itH}\phi_b
\right|^2
\right\rangle 
dt\\
&=
b^{-7}Q_1
\int_{\mathbb{R}}
\left\langle
\frac{\exp(-\mu_1 r)}{r}\ast 
\left|
\left(
e^{-itb^{-2}H(b)}\phi
\right)_b
\right|^2
, 
\left|
\left(
e^{-itb^{-2}H(b)}\phi
\right)_b
\right|^2
\right\rangle 
dt\\
&=
b^{-7+2+3+3}Q_1
\int_{\mathbb{R}}
\left\langle
\frac{\exp(-b\mu_1 r)}{br}\ast 
\left|
e^{-itH(b)}\phi
\right|^2
, 
\left|
e^{-itH(b)}\phi
\right|^2
\right\rangle 
dt\\
&=
Q_1
\int_{\mathbb{R}}
\left\langle
\frac{\exp(-\sqrt{|Q_1|}r)}{r}\ast 
\left|
e^{-itH(b)}\phi
\right|^2, 
\left|
e^{-itH(b)}\phi
\right|^2
\right\rangle
dt \\
&=
\Psi_1(Q_1).
\end{align*}
By the Plancherel theorem, 
we have 
\begin{align*}
\Psi(\alpha) 
=
4\pi
\int_{\mathbb{R}}
\int_{\mathbb{R}^3}
\frac{\alpha}{|\alpha|+|\xi|^2}
\left|
\left(
\mathfrak{F}
\left|
e^{-itH(b)}\phi
\right|^2
\right)
(\xi)
\right|^2
d\xi dt,
\end{align*} 
where $\mathfrak{F}$ denotes 
the Fourier transform on $\mathcal{H}$:
\[
\mathfrak{F}\varphi(\xi)
:=
(2\pi)^{-3/2}
\int_{\mathbb{R}^3}
e^{-ix\cdot \xi}\varphi(x)dx, 
\quad
\varphi\in L^1\cap \mathcal{H}.
\]
Therefore, 
$\Psi:\mathbb{R}\to\Psi(\mathbb{R})$ 
is odd, continuous, bijective  
and monotonically increasing. 
Thus, we obtain 
\[
\Psi\left( 
m_0+\sum_{k=1}^j
\frac{q_k}{2^k}
\right)
\le 
\Psi(|Q_1|)
<
\Psi\left( 
m_0+\sum_{k=1}^j
\frac{q_k}{2^k}
+\frac{1}{2^j}
\right) 
\]
and 
\[
m_0+\sum_{k=1}^j
\frac{q_k}{2^k}
\le 
|Q_1|
<
m_0+\sum_{k=1}^j
\frac{q_k}{2^k}
+\frac{1}{2^j}.
\]
Hence (\ref{recon-2}) holds. 
\end{proof}
\section{Application}\label{sec:application}
In this section, 
we consider the inverse scattering problem 
for the semi-relativistic Hartree equation 
\begin{equation}\label{SRH}
\Bigl(
i\partial_t +\sqrt{1-\Delta}
\Bigr)
w=
F_2(w),\quad 
(t,x)\in {\mathbb R}^{1+3}.\tag{SRH}
\end{equation}
Here, 
\[
F_2(w)=\biggl(
Q_2 \frac{\exp(-\mu_2 r)}{r} \ast |w|^2
\biggr) w.
\]
The equation (\ref{SRH}) is used to 
describe Boson stars. 
For the detailed physical background, 
see Lenzmann \cite{Lenzmann}. 

There is no result for 
the inverse scattering problem 
for the nonlinear semi-relativistic equation.  
Instead, we review 
the inverse scattering problem 
for the nonlinear Klein-Gordon equation. 
Morawetz--Strauss \cite{Morawetz-Strauss} 
initially studied  
the inverse scattering problem for 
the Klein-Gordon equation 
with power nonlinearity.  
Later, Bachelot \cite{Bachelot} 
considered more general cases. 
Weder \cite{Weder00-1,Weder02} proved that  
a more general class of nonlinearities 
is uniquely reconstructed, 
and moreover, 
a method is given for the unique reconstruction 
of the potential that acts as a linear operator 
and that this problem was not considered 
in \cite{Morawetz-Strauss,Bachelot}.
The inverse scattering problem for 
the Klein-Gordon equation 
with a cubic convolution 
\begin{align*}
\partial_t^2 w -\Delta w +w =
\bigl( V\ast |w|^2 \bigr) w, 
\quad 
\text{in $(t,x)\in \mathbb{R}^{1+n}$} 
\end{align*} 
was initially studied by \cite{Sasaki-Watanabe}. 
In the case where $V$ satisfies
$V=\lambda |x|^{-\sigma}$ 
for some $\sigma$ and $\lambda$, 
\cite{Sasaki 2007} proved that 
$V$ can be recovered. 
Unfortunately, 
as far as the author knows, 
there is no known method  
to recover the nonlinearity $F_2(w)$. 

We shall  
determine the value of $Q_2$ and $\mu_2$  
from the knowledge of the scattering operator 
given by the following Proposition: 
\begin{prop}\label{prop:direct2}
Let $s\ge 5/6$. 
Put $U_2(t)=e^{-it\sqrt{1-\Delta}}$,  
$X_2=H^{s}$, 
$Y_2=L^2(\mathbb{R};H_6^{s-5/6}(\mathbb{R}^3))$ 
and 
$Z_2=C(\mathbb{R};X_2)\cap Y_2$.
Then there exists some $\delta>0$ 
satisfying the following properties: 

If $\phi_-\in B(\delta;X_2)$, 
then there uniquely exist 
$w\in Z_2$ and $\phi_+\in X_2$ 
such that 
\begin{align}
w(t) &=U_2(t)\phi_- 
+
\frac{1}{i}\int^{t}_{-\infty}
U_2(t-\tau)F_2(w(\tau))d\tau ,\\
\phi_+ &=\phi_- +
\frac{1}{i}\int_{\mathbb{R}}
U_2(-t)F_2(w(t))dt ,\\
\| w\|_{Z_2}
&\le C\| 
\phi_-\|_{X_2},\\
\| w-U_2(t)\phi_-\|_{Z_2}
&\le C\| 
\phi_-\|_{X_2}^3,\\
\lim_{t\to\pm\infty}&
\| w(t)-U_2(t)\phi_\pm\|_{X_2}=0.
\end{align} 
Therefore, 
we can define the scattering operator 
for (\ref{SRH})
\[
S_2:B(\delta;X_2)\ni \phi_- 
\mapsto \phi_+ \in X_2 .
\]
\end{prop}
\begin{rem}  
We can easily show 
Proposition \ref{prop:direct2} 
by following the proof of 
Theorem 3.4 in 
Cho--Ozawa \cite{Cho-Ozawa}.
\end{rem}
From the knowledge of $(\phi_-,S_2(\phi_-))$,  
we give the following formula 
for determining $Q_2/\mu_2^2$:  
\begin{theorem}\label{mainthm2}
Let $s$ be a positive number given by 
Proposition \ref{prop:direct2}. 
Let $1\le p<12/7$ and $k>11/12$. 
Assume that 
\[
\phi \in 
\bigl( 
H^s(\mathbb{R}^3) \cap H^k_p(\mathbb{R}^3)
\bigr) \setminus \{ 0 \} .
\]
\begin{enumerate}[(i)]
  \item we have 
\begin{equation}\label{recon-3}
\frac{Q_2}{\mu_2^2}=
\frac{
\displaystyle{
\lim_{\lambda\to\infty}
i\lambda^4
\Big\langle 
(S_2 -id)
(\lambda^{-3}\phi_\lambda ), 
\phi_\lambda
\Big\rangle
}}{\displaystyle{
4\pi 
\| e^{i\frac{1}{2}\Delta} \phi}\|_{(4,4)}^4} .
\end{equation}
  \item Put 
\begin{align*}
d&=
\left|
\frac{Q_2}{\mu_2^2}
\right|^{1/2},\\
\Psi_2(\alpha)&=
\int_{\mathbb{R}}
\left\langle
\frac{\alpha\exp(-\sqrt{|\alpha|}r)}{r}\ast 
\left|
e^{-it\sqrt{d^2-\Delta}}\phi
\right|^2, 
\left|
e^{-it\sqrt{d^2-\Delta}}\phi
\right|^2
\right\rangle
dt,
\quad
\alpha\in\mathbb{R},\\
h&=
\lim_{\varepsilon\to 0}
i\varepsilon^{-3}d^{-6}
\left\langle
(S_2 -id)(\varepsilon\phi_d), 
\phi_d
\right\rangle, \\
l_0&=
\max\left\{
l\in\mathbb{Z}_{\ge 0};
\Psi_2(l)\le |h|
\right\},\\
p_1&=
\max\left\{
p=0,1;
\Psi_2\left(
l_0+\frac{p}{2}
\right)
\le |h|
\right\},\\
p_{j+1}&=
\max\left\{
p=0,1;
\Psi_2\left(
l_0+\sum_{k=1}^j\frac{p_k}{2^k}+\frac{p}{2^{j+1}}
\right)
\le |h|
\right\},\quad j=1,2,\cdots .
\end{align*} 
Then we have 
\begin{align}\label{recon-4}
Q_2=\mathrm{sign}
\left( \frac{Q_2}{\mu_2^2}\right) 
\left(
l_0+\sum_{j=1}^\infty\frac{p_j}{2^j}
\right) .
\end{align}
\end{enumerate}
\end{theorem}
\subsection{Proof of Theorem \ref{mainthm2}}
In order to show Theorem \ref{mainthm2}, 
we first prepare the following lemma:
\begin{lem}\label{lem:NRL}
For $\lambda>0$, 
let 
\[
U^\lambda(t) =
e^{it\lambda^2 -it\lambda\sqrt{\lambda^2-\Delta}}, 
\quad 
U^\infty(t) = 
e^{i\frac{1}{2}\Delta}.
\]
Assume that $1\le p<12/7$ and $k>11/12$. 
If $\phi\in H^s(\mathbb{R}^3)\cap H^k_p(\mathbb{R}^3)$, 
then we have 
\begin{equation}\label{NRL}
\lim_{\lambda\to\infty}
\| U^\lambda(t)\phi - 
   U^\infty(t)\phi \|_{(4,4)}=0.
\end{equation}
\end{lem}
\begin{proof}
From the embedding 
$H^{3/4}(\mathbb{R}^3)\hookrightarrow L^4(\mathbb{R}^3)$ 
and 
the Plancherel theorem 
we obtain 
\begin{align*}
\| U^\lambda(t)\phi - 
   U^\infty(t)\phi \|_4
&\le C
\| U^\lambda(t)\phi - 
   U^\infty(t)\phi \|_{H^{3/4}(\mathbb{R}^3)} \\
&\le C
\Bigl\|
\Bigl(
e^{it\lambda^2 -it\lambda\sqrt{\lambda^2+|\xi|^2}}
-
e^{-it|\xi|^2/2}
\Bigr) 
\left\langle \xi \right\rangle^{3/4}
{\mathcal F}\phi
\Bigl\|_{L^2(\mathbb{R}^3_\xi)} . 
\end{align*} 
Since 
$\left\langle \xi \right\rangle^{3/4}
{\mathcal F}\phi \in {\mathcal H}$ 
and 
\[
\lambda^2 -\lambda\sqrt{\lambda^2+|\xi|^2} 
= 
\frac{-|\xi|^2}{
1+\sqrt{1+|\xi|^2/\lambda^2} 
}
\to 
\frac{-|\xi|^2}{2}
\quad\text{as $\lambda\to\infty$}
\]
for any $\xi\in\mathbb{R}^3$, 
it follows from the Lebesgue dominated 
theorem that 
\begin{equation}\label{NRL-1}
\lim_{\lambda\to\infty}
\| U^\lambda(t)\phi - 
   U^\infty(t)\phi \|_4=0
\quad\text{for any $t\in\mathbb{R}$}.
\end{equation}
Now we put $\lambda>1$. 
By the $L^p-L^q$ estimate 
for the free Klein-Gordon equation 
in \cite{Brenner}, we obtain 
\begin{align}
\| U^\lambda(t)\phi\|_4 
&= 
\Bigl\|
\bigl(
e^{-it\lambda^2\sqrt{1-\Delta}}
\phi_\lambda
\bigr)_{\lambda^{-1}} 
\Bigr\|_4 \nonumber\\
&= 
\lambda^{-3/4} 
\| e^{-it\lambda^2\sqrt{1-\Delta}}
\phi_\lambda \|_4 \nonumber\\
&\le C
\lambda^{-3/4} 
|t\lambda^2|^{-3/4}
\| \phi\|_{H^{5/4}_{4/3}(\mathbb{R}^3)} \nonumber\\
&\le C
|t|^{-3/4}
\| \phi\|_{H^{5/4}_{4/3}(\mathbb{R}^3)} \label{LpLq}.
\end{align} 
Using the complex interpolation method 
for the linear operator $U^\lambda(t)$, 
we see from 
\[
\| U^\lambda(t)\phi\|_4 \le C 
\| U^\lambda(t)\phi\|_{H^{3/4}(\mathbb{R}^3)} \le C
\| \phi\|_{H^{3/4}(\mathbb{R}^3)}
\] 
and (\ref{LpLq}) that 
\begin{equation}\label{interpolation}
\| U^\lambda(t)\phi\|_4 \le C  
(1+|t|)^{-3\theta/4}\| \phi\|_{A_\theta}, 
\end{equation}
where 
\[
A_\theta=H^{3/4}(\mathbb{R}^3)\cap 
H^{k_\theta}_{p_\theta}(\mathbb{R}^3), 
\quad 
k_\theta=\frac{3}{4}+\frac{\theta}{2}, 
\quad  
p_\theta=\frac{1}{2}+\frac{\theta}{4}, 
\quad 
0\le \theta\le 1 .
\] 
Thus, we can easily see that 
the left hand side of 
(\ref{interpolation}) belongs 
$L^4(\mathbb{R})$ if $1/3<\theta\le 1$. 
Therefore, we obtain 
\[
\| U^\lambda(t)\phi - U^\infty(t)\phi \|_4 
\le C g(t),
\]
where $g\in L^4(\mathbb{R})$ 
is some suitable function 
independent of $\lambda$. 
By (\ref{NRL-1}), it follows from 
the Lebesgue dominated theorem 
with respect to time $t$ that (\ref{NRL}) holds. 
\end{proof}
We are ready to prove  
Theorem \ref{mainthm2}. 
\begin{proof}[Proof of Theorem \ref{mainthm2}]
Following the line of the proof of 
(\ref{recon-1}), 
we obtain 
\begin{align*}
\lim_{\lambda\to\infty}&
i\lambda^4 
\Big\langle 
(S_2-id)(\lambda^{-3}\phi_\lambda),
\phi_\lambda
\Big\rangle \\
&=
\lim_{\lambda\to\infty}
i\lambda^{-5} 
\int_{\mathbb{R}}
\Big\langle 
F_2(U_2(t)\phi_\lambda), 
U_2(t)\phi_\lambda
\Big\rangle dt\\
&=
\lim_{\lambda\to\infty}
\int_{\mathbb{R}^{1+2\cdot 3}}
Q_2
\frac{\exp(-\mu_2 |y|)}{|y|}
\bigl|
U^\lambda(t)\phi(x-\lambda^{-1}y)
\bigl|^2 
\bigl|
U^\lambda(t)\phi(x)
\bigl|^2 
d(t,x,y) .
\end{align*}
By (\ref{NRL}), we have (\ref{recon-3}). 
The remaining formula (\ref{recon-4}) can be shown 
by the same argument as the proof of (\ref{recon-3}).
\end{proof}
\appendix
\section{Some Norms of the Yukawa Potential}\label{Appendix}
In this appendix, 
we consider some norms of the Yukawa potential $e^{-r}/r$.
Our claim is the following: 
\begin{prop}\label{norm-of-Yukawa}
\begin{enumerate}[(i)]
  \item The condition (\ref{RK}) is equivalent to (\ref{RK-e}).
  \item We have (\ref{C-b3/2}).
  \item We have $\| e^{-r}/r\|_1=4\pi$.
\end{enumerate}
\begin{proof}
For $1\le p<3$, we obtain 
\[
\left\|
\frac{e^{-r}}{r}
\right\|_p
=
\left(
4\pi p^{p-3}\Gamma(3-p)
\right)^{1/p}.
\]
Here, $\Gamma$ is the usual Gamma function. 
In particular, we see that 
\[
\left\|
\frac{e^{-r}}{r}
\right\|_1
=
4\pi,\qquad
\left\|
\frac{e^{-r}}{r}
\right\|_{3/2}
=
\frac{2^{5/3}\pi}{3}.
\]
Hence we have proved (iii).
By Aubin and Talenti \cite{Aubin,Talenti}, 
the best constant $C_b$ is explicitly given by 
\[
C_b=\frac{1}{3\pi}
\left(
\frac{4}{\sqrt{\pi}}
\right)^{2/3}.
\] 
Thus, we have 
\[
C_b^2\left\|
\frac{e^{-r}}{r}
\right\|_{3/2} 
=
\frac{8\pi^{-1/3}}{9}<1,
\]
which implies (ii). 

We now consider the Rollnik norm of the Yukawa potential. 
Lieb \cite{Lieb} proved that 
the best constant 
for the Hardy-Littlewood-Sobolev inequality 
is given by 
\[
\sup\left\{
\frac{\| \psi\ast r^{-2}\|_3}{\| \psi\|_{3/2}};\ 
\psi\in L^{3/2}(\mathbb{R}^3)\setminus \{0\}
\right\}
=
2^{2/3}\pi^{4/3}.
\]
Therefore, we see from the H\"{o}lder inequality that 
\begin{align}
\left\|
\frac{e^{-r}}{r}
\right\|_R
&=
\sqrt{
\left\langle 
\frac{e^{-r}}{r}\ast\frac{1}{r^2},
\frac{e^{-r}}{r}
\right\rangle
}
\le 
\left\|
\frac{e^{-r}}{r}
\right\|_{3/2}^{1/2}
\left\|
\frac{e^{-r}}{r}\ast\frac{1}{r^2}
\right\|_{3}^{1/2}\nonumber\\
&\le 
2^{1/3}\pi^{2/3}
\left\|
\frac{e^{-r}}{r}
\right\|_{3/2}
=
4\pi \frac{\pi^{2/3}}{3}
<4\pi .\label{Rollnik-Yukawa}
\end{align} 

On the other hand, 
we next consider $\| e^{-r}/r\|_{\mathcal{K}}$.
Since the function $(e^{-r}/r)\ast r^{-1}$ is radial, 
we have 
\begin{align*}
\frac{e^{-r}}{r}\ast \frac{1}{r}(x)
&=
\int_{\mathbb{R}^3}
\frac{\displaystyle{e^{-|y|}}}{\displaystyle{
|y|\sqrt{||x|-y_1|^2+|y_2|^2+|y_3|^2}}}
dy\\
&=
\int_{0}^{\infty}
\int_{R\mathbb{S}^2}
\frac{\displaystyle{
e^{-R}}
}{\displaystyle{
R\sqrt{||x|-\theta_1|^2+|\theta_2|^2+|\theta_3|^2}
}}
d\sigma(\theta)
dR\\
&=
\int_{0}^{\infty}
R^{-1}e^{-R}
\int_{R\mathbb{S}^2}
\frac{d\sigma(\theta)}{\displaystyle{
\sqrt{||x|-\theta_1|^2+|\theta_2|^2+|\theta_3|^2}
}}
dR.\\
&=
\int_{0}^{\infty}
R^{-1}e^{-R}
\int_{-R}^{R}
\int_{\sqrt{R^2-s^2}\mathbb{S}^1}
\frac{d\sigma(\vartheta)}{\displaystyle{
\sqrt{||x|-s|^2+R^2-s^2}
}}
\frac{Rds}{\sqrt{R^2-s^2}}dR.\\
&=
4\pi
\int_{0}^{\infty}
e^{-R}
\frac{\min\{ R,|x|\}}{|x|}
dR
=
\frac{4\pi}{|x|}
\left(
1-e^{-|x|}
\right) .
\end{align*} 
Thus, we obtain 
\begin{align}\label{GK-Yukawa}
\left\|
\frac{e^{-r}}{r}
\right\|_{\mathcal{K}}
=
4\pi .
\end{align}
Hence (ii) is true.
\end{proof}
\end{prop}
\subsection*{Acknowledgments}
The author would like to acknowledge 
the helpful advice of Dr. Itaru Sasaki.  
The author is grateful to the referee for 
pointing out some gaps in the manuscript. 

\end{document}